\documentclass[12pt]{amsart}
\usepackage{array}
\usepackage{amstext}
\usepackage{amsfonts}
\usepackage{amsmath}
\usepackage{filecontents}

\usepackage{amssymb}
\usepackage{tikz}
\usepackage[colorlinks=true, linkcolor=blue]{hyperref}

\usetikzlibrary{arrows, positioning}
\tikzset{Mylong/.style={text width=3.1cm, align=center}, myarr/.style={->, double equal sign distance, -implies}}

\newtheorem*{theorem*}{Theorem} 
\newtheorem{theorem}{Theorem}[section]

\newtheorem{lemma}[theorem]{Lemma}
\newtheorem*{lemma*}{Lemma}
\newtheorem{proposition}[theorem]{Proposition}

\theoremstyle{definition}
\newtheorem{definition}[theorem]{Definition}
\newtheorem{example}[theorem]{Example}
\theoremstyle{remark}
\newtheorem*{remark}{Remark}

\newcommand{\Hom}{\ensuremath{\operatorname{Hom}}}
\newcommand{\ke}{\ensuremath{\operatorname{Ker}}}

\title[On aperiodicity and hypercyclic weighted translation operators] {On aperiodicity and hypercyclic weighted translation operators}
\author[Kui-Yo Chen]{Kui-Yo Chen}


\thanks{This work is partially supproted by MOST 104-2115-M-001-007 and 105-2115-M-001-003.}
\address{Department of Mathematics, National Taiwan University, Taiwan}
\email{r04221001@ntu.edu.tw}

\date{\today}
\begin{document}

\begin{abstract}
We give several equivalent characterization of aperiodicity of an element on locally compact group $G$, and give an intuition for ``How strong does the aperiodicity of an element affect the existence of hypercyclic weighted translation operators?''. In fact, if $a$ is an aperiodic element in $G$, then there exists a mixing, chaotic and frequently hypercyclic weighted translation $T_{a,w}$ on $L^p(G)$.
\end{abstract}

\maketitle
\addcontentsline{toc}{section}{Title}
\baselineskip17pt

\section{Introduction} 
\label{sec:introduction}
This paper has two parts. Firstly, we characterize some equivalent statements of aperiodicity of an element on locally compact group. Secondly, we use those equivalent statements of aperiodicity to give the existence of hypercyclic weighted translations. For some cases we actually can find expilcit form of hypercyclic weighted translations.

In the field of linear chaos, people focus on the linear operators which act on a Banach space and discuss their dynamic properties like hypercyclicity and chaoticity. For our discussion, we focus on \cite{Hypercyclic_on_groups}, see also \cite{Chaotic_on_groups, non-torsion} which characterize the chaoticity and hypercyclicity of a weighted translation operator on the $L^p$ space of a locally compact group.

An operator $T$ on a Banach space $X$ is called {\it hypercyclic} if there exists a vector $x\in X$ such that its orbit is dense in the whole space (i.e. $orb(T,x):=\left\{T^nx|n\in \mathbb{N}\right\}$ is dense in $X$). An operator $T$ is called {\it weakly mixing} if $T\oplus T$ defined on $X\times X$ is hypercyclic. An operator $T$ is called {\it mixing} if for any nonempty opens $U,V$ in $X$, there exists $N\in \mathbb{N}$ such that $T^nU\cap V\neq \varnothing$ for all $n>N$. An operator $T$ is called {\it chaotic} if it is hypercyclic and the set of periodic points is dense. An operator $T$ is called {\it frequently hypercyclic} if there is some $x \in X$ such that for any nonempty open subset $U$ of $X$, $n_k= O(k)$, where $n_k$ is a strictly increasing sequence of integers such that $T^{n_k}x$ is $k$-th element lying in $U$ (by \cite[Proposition 9.3. p.237]{Linear_chaos}, this is an equivalent statement of \cite[Definition 9.2. p.237]{Linear_chaos}). (Note that we only consider the case that the operator is a weighted translation in this paper).

\begin{theorem*}
(Frequent Hypercyclicity Criterion, \cite[Theorem 9.9 and Proposition 9.11.]{Linear_chaos}). Let $T$ be an operator on a separable Fr\'echet space $X$. If there is a dense subset $X_0$ of $X$ and a map $S : X_0 \to X_0$ such that, for any $x\in X_0$,
\begin{enumerate}
\item
$\sum\limits ^{\infty}_{n=0}T^nx$ converges unconditionally,
\item
$\sum\limits ^{\infty}_{n=0}S^nx$ converges unconditionally,
\item
$TSx = x$,
\end{enumerate}
then $T$ is frequently hypercyclic. Moreover, $T$ is also chaotic and mixing. In particular, it is also weakly mixing and hypercyclic.
\end{theorem*}

The graph below characterizes some relations between these dynamical properties which have been discussed, see \cite{Linear_chaos}:
\vskip1em
\hskip-2em
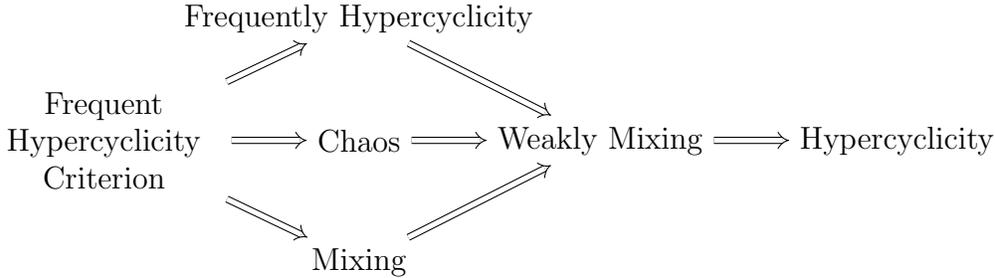
\begin{tikzpicture}
\node (conI) {Frequently Hypercyclicity};
\node (conII) [below = of conI] {Chaos};
\node (conIII) [Mylong, left = of conII] {Frequent\break Hypercyclicity Criterion};
\node (conIV) [right = of conII] {Weakly Mixing};
\node (conV) [below = of conII] {Mixing};
\node (conVI) [right = of conIV] {Hypercyclicity};

\draw [myarr] (conIII) -- (conII);
\draw [myarr] (conII) -- (conIV);
\draw [myarr] (conI) -- (conIV);
\draw [myarr] (conIII) -- (conI);
\draw [myarr] (conIII) -- (conV);
\draw [myarr] (conV) -- (conIV);
\draw [myarr] (conIV) -- (conVI);
\end{tikzpicture}
\vskip1em

In the above, we describe the general conclusions. In this article, we only focus on the weighted translation operators.

Let $G$ be a locally compact group and $a$ be an element of $G$. The weighted translation operator $T_{a,w}$ is a bounded linear self-map on the Banach space $L^p(G)$ (by using the right Haar measure on $G$), for some $p \in [1, \infty)$, defined by 
\[T_{a,w}(f)(x):=w(x)f(xa^{-1}),\]
where the weight $w$ is a bounded continuous function from $G$ to $(0,\infty )$. We denote $T_{a,1}$ by $T_a$ so that $T_a w$ is the function $w$ translation by $a$, while $T_{a,w}$ is a weighted translation operator.

To analyze $T_{a,w}$, we would like to classify some different topological properties of the elements in $G$. We call an element $a$ of $G$ {\it torsion} if it has finite order. An element $a$ is {\it periodic} if the closed subgroup $G(a)$ generated by $a$ (i.e. $G(a)=\overline{<a>}$) is compact in $G$. An element $a$ is {\it aperiodic} if it is not periodic.

Note that we don’t need $G$ to be Hausdorff in this paper, unless we say that $G$ is a Hausdorff group.

\begin{lemma*} $ $
\cite[Lemma 2.1., C. CHEN AND C-H. CHU]{Hypercyclic_on_groups} An element $a$ in a second countable group $G$ is aperiodic if, and only
if, for each compact subset $K\subseteq G$, there exists $N \in\mathbb{N}$ such that $K\cap Ka^n=\varnothing$ for
$n > N$.
\end{lemma*}

In \cite[Lemma 2.1]{Hypercyclic_on_groups}, they give an equivalence statement of aperiodicity when $G$ is a second countable locally compact Hausdorff group. We give another equivalence statement when $G$ is second countable (Theorem \ref{equaperterminal}), which we define in Definition \ref{defterminal} and we call it {\it terminal pair}.


\begin{lemma*} $ $
\cite[Lemma 1.1., C. CHEN AND C-H. CHU]{Hypercyclic_on_groups} Let $G$ be a locally compact group and let $a\in G$ be a torsion element.
Then any weighted translation $T_{a,w} : L^p(G) \to L^p(G)$ is not hypercyclic, for
$1 \le p <\infty$.
\end{lemma*}

On the other hand, \cite[Lemma 1.1]{Hypercyclic_on_groups} gives the non-existence of hypercyclic weighted translations when the element $a$ is torsion. So the question of the existence of hypercyclic weighted translations will be focused on the case in which $a$ is non-torsion, this is, $a$ is either non-torsion periodic or $a$ is aperiodic.

All examples of hypercyclic weighted translation operators in previously known literature are all associated to aperiodic $a$.
One of the most important concrete examples is the weighted backward shift operator on ${\ell}^p(\mathbb{Z})$ \cite[Example 4.15. p.102]{Linear_chaos} and there are also some classical analogous examples relevant to semigroups \cite{semigroups}; in fact, it can correspond to our case for several admissible weights by conjugation.
So we unify them to our Main Theorem \ref{main}.

\begin{theorem*}[Main Theorem]
Let $G$ be a second countable locally compact group and $a$ be an aperiodic element in $G$, then there exists a weighted translation operator $T_{a,w}$ which is mixing, choatic and frequently hypercyclic on $L^p(G)$ for all $p\in [1,\infty)$, simultaneously.
\end{theorem*}


\section{Equivalent statements of aperiodicity} 
\label{sec:equivalent_statements_of_aperiodicity}

\begin{proposition}
\label{homoaper}
Let $G,G'$ be locally compact groups, $\phi (a)$ is an aperiodic element of $G'$, where $\phi :G \to G'$ is a continuous homomorphism, then $a$ is also an aperiodic element of $G$. In other words, continuous homomorphisms pullback the aperiodicity.
\end{proposition}
\begin{proof}
If $a$ is periodic, then $\overline{<a>}$ is a compact group and so does $\phi(\overline{<a>})$, but $\phi (a)\in \phi(\overline{<a>})$, contradiction as there are no aperiodic element in compact group.
\end{proof}

Let $G$ be a topological group. Define the {\it Hausdorffication} of $G$ by the natural continuous quotient map $\pi :G\to \widetilde{G}$, where $\widetilde{G}:=G/\overline{\{e\}}$ and $e$ denotes the identity element of $G$. (This may be related to the ``Hausdorffication'' which is defined as a left adjoint of forget functor in general topology). The reason that we consider the Hausdorffication is lots of statements in this paper without the assumption that $G$ is Hausdorff; but in fact, we first assume $G$ is Hausdorff in proving. Next, we prove the non-Hausdorff case by considering its Hausdorffication, and show it can preserve or pullback some conclusion we want. So we will claim several arguments below:
\begin{enumerate}
\item
Each open or closed subset of $G$ is a union of the cosets of $\overline{\{e\}}$.

Since $\overline{S}x=\overline{Sx}$ for any $S\subseteq G$ and $x\in G$. Choose $S=\{e\}$ and $x\in \overline{\{e\}}$, then  $\overline{\{x\}}=\overline{\{e\}x}=\overline{\{e\}}x=\overline{\{e\}}$ (last equality follows by $x\in \overline{\{e\}}$ and $\overline{\{e\}}$ is a subgroup of $G$), which implies $\overline{\{e\}}$ is indiscrete topology space and so does all cosets of $\overline{\{e\}}$. This implies the statement we want.

\item
There's a one to one correspondence between the set of open(closed) of $G$ and $\widetilde{G}$'s. In particular, $G$ is first(second) countable if and only if $\widetilde{G}$ also, and $\widetilde{G}$ is Hausdorff.

This correspondence is given by $(U \mapsto \pi (U))$ and $(\widetilde{U}\mapsto \pi^{-1}(\widetilde{U}))$ for $U$ is open in $G$ and $\widetilde{U}$ is open in $\widetilde{G}$. To verify that these two maps compose to identity for two sides, we only need to say $\pi^{-1}\pi(U)\subseteq U$ for any $U$ is open in $G$, since the others are relatively obvious. Assume there exists $x\in \pi^{-1}\pi(U)\setminus U$, then $\pi(x)\in \pi(U)$, so there is some $y\in U$ such that $\pi(x)=\pi(y)$ (so $\pi(xy^{-1})=\pi(e)$, then $xy^{-1}\in \ke \pi =\overline{\{e\}}$, then $x\in \overline{\{e\}}y$). But by (1), $\overline{\{e\}}y\subseteq U$, since $y\in U$, a contradiction as this implies $x\in U \cap \pi^{-1}\pi(U)\setminus U=\varnothing$.

\item
$\pi$ is an open, closed and proper mapping. In particular, $G$ is locally compact iff $\widetilde{G}$ also. On the other hand, the one to one correspondence in (2), not only just closed sets, but also closed compact sets (since $\pi$ is proper).

The openness and closedness follow by (1) and (2) immediately. For the properness, let $\widetilde{K}$ be compact in $\widetilde{G}$, and we'll check $\pi^{-1}(\widetilde{K})$ also. Let $\{U_{\alpha}\}$ be a open covering of $\pi^{-1}(\widetilde{K})$, so by correspondence, $\{\pi(U_{\alpha})\}$ be a open covering of $\widetilde{K}$. Hence induce a finite subcovering $\{\pi(U_{i})\}$, then we can check $\{U_{i}\}$ is actually a finite subcovering of $\pi^{-1}(\widetilde{K})$.

\item
Let $Y$ be an arbitrary Hausdorff topology space, then $\Hom(G,Y)\cong \Hom(\widetilde{G},Y)$. (i.e. There's a natural one to one correspondence between the set of continuous functions from $G$ to $Y$ and $\widetilde{G}$'s.)

Given $w\in \Hom(G,Y)$, define $\widetilde{w}(\widetilde{x}):=w(x)$ where $\widetilde{x}=\overline{\{e\}}x$, then $\widetilde{w}$ is a well-defined continuous function on $\widetilde{G}$. On the other hand, Given $\widetilde{w}\in \Hom(\widetilde{G},Y)$, then we get $w:=\widetilde{w} \circ \pi$, which is also a well-defined continuous function on $G$. It easy to check these two mappings between $\Hom(G,Y)$ and $\Hom(\widetilde{G},Y)$ are inverse to each others. (Equivalently, one can say that any continuous function from $G$ to $Y$ factor through $\widetilde{G}$. In other words, this is the universal property of Hausdorffication.)

\item
$a$ is an aperiodic in $G$ iff $\pi(a)$ is an aperiodic in $\widetilde{G}$.

If $\pi(a)$ is a periodic, then $\overline{<\pi (a)>}$ is compact and so does $\pi^{-1}(\overline{<\pi (a)>})$, since $\pi $ is proper, but $a\in \pi^{-1}(\overline{<\pi (a)>})$. There is a contradiction. The other side follows from Proposition \ref{homoaper} immediately.
\end{enumerate}

The idea of the proofs of Proposition \ref{runawayaper}, Lemmas \ref{secondcompact} and \ref{discrete} and Proposition \ref{aperrunaway} follows from \cite[Lemma 2.1]{Hypercyclic_on_groups}.

\begin{proposition}
\label{runawayaper}
Let $G$ be a locally compact group, $a\in G$ has the following properties: For any compact subset $K$ of $G$, there exists $N\in \mathbb{N}$ such that $K\cap Ka^n = \varnothing$ for $n>N$. Then $a$ is an aperiodic element in $G$.
\end{proposition}

\begin{proof}
Suppose $a$ is a periodic element, so the closed subgroup $G(a)$ generated by $a$ is compact. Now set $K=G(a)$, then we have $K\cap Ka^n =G(a) \neq \varnothing$ for all $n\in \mathbb{Z}$, which is a negative statement of the condition.
\end{proof}

\begin{lemma}
\label{secondcompact}
Let $G$ be a first countable locally compact Hausdorff group with $a\in G$ being an aperiodic element, then $G(a)$ is a second countable compactly generated abelian group.
\end{lemma}

\begin{proof}
Since $a$ is an aperiodic element, so $G(a)$ is non-compact closed abelian subgroup of $G$ (the commutativity follows by net arguement and $G(a)$ is Hausdorff.). Moreover, by \cite[Theorem 5.14]{E_Hewitt}, there exists a compactly generated subgroup $G'$ of $G$ which containing $G(a)$, since $\left\{e,a\right\}$ is a compact subset of $G$.

First, we'll check that it is second countable. Since $G(a)$ is a first countable locally compact Hausdorff group, hence metrizable \cite[Birkhoff-Kakutani metrizable theorem]{Birkhoff}. On the other hand, since the set $\left\{a^j\right\}_{j\in \mathbb{Z}}$ is dense in $G(a)$, so $G(a)$ is a separable metrizable space, hence second countable.

So in the last step, we'll check that it is compactly generated. There's a important thing that not every subgroup of a compactly generated group is compactly generated (even a closed subgroup), but in our case it works. Let $G'$ be generated by the compact set $K_0$, where $K_0=\overline{V_0}$ for some open $V_0$ containing $\left\{e,a\right\}$ (the existence follows from the proof in \cite[Theorem 5.14]{E_Hewitt}). We'll claim that $G(a)$ is generated by the compact set $K_1$, where $K_1=\overline{V_1}$ and $V_1=V_0\cap G(a)$. Since the set $\left\{a^j\right\}_{j\in \mathbb{Z}}$ is dense in $G(a)$, which implies $\left\{V_1a^j\right\}_{j\in \mathbb{Z}}$ is a covering of $G(a)$ (we'll prove this in the Remark below). So $G(a)\subseteq \cup_{j\in \mathbb{Z}}V_1a^j\subseteq \cup_{j\in \mathbb{Z}}K_1a^j\subseteq G(a)$ (last step follows from $K_1\subseteq G(a)$), that means that any $x\in G(a)$ can write it as the form $ka^j$ for some $k\in K_1$ and some $j\in \mathbb{Z}$ (i.e. $G(a)$ is compactly generated, since $a$ is also in $K_1$).
\end{proof}
\begin{remark}
To prove that $\left\{V_1a^j\right\}_{j\in \mathbb{Z}}$ is a covering of $G(a)$, we need to say that every $x\in G(a)$ which is a limit of a subsequence $\left\{a^{n_k}\right\}$ has been covered by $\left\{V_1a^{n_k}\right\}$ (we view $V_1$ as the relative open neighborhood of $e$ in $G(a)$ here). Now choose $V_2$ the symmetric open neighborhood of $e$ in $V_1$. Then $a^{n_k}\in V_2x$ for $k$ large enough, hence $x\in V_2^{-1}a^{n_k}=V_2a^{n_k}\subseteq V_1a^{n_k}$, so we are done.
\end{remark}

\begin{lemma}
\label{discrete}
Let $G$ be a first countable locally compact Hausdorff group with $a\in G$ being an aperiodic element, then $G(a)$ is topologically isomorphic to $\mathbb{Z}$.
\end{lemma}

\begin{proof}
So by previous lemma and \cite[Theorem 9.8]{E_Hewitt}, $G(a)\cong \mathbb{R}^n\times \mathbb{Z}^m\times \mathbb{F}$ for some $n,m\in \mathbb{N}$ and $\mathbb{F}$ is a compact group, and $a$ identifies with the element in $\mathbb{R}^n\times \mathbb{Z}^m\times \mathbb{F}\setminus \left(\{0\}\times \{0\}\times \mathbb{F}\right)$ and $<a>$ the cyclic subgroup generated by $a$ will not have any accumulation points. That means $G(a)$ is actually a discrete group, hence $G(a)=<a>$, which is also isomorphic to $\mathbb{Z}$.
\end{proof}

\begin{proposition}
\label{aperrunaway}
Let $G$ be a first countable locally compact group with $a\in G$ being an aperiodic element, then for any compact subset $K$ of $G$, there exists $N\in \mathbb{N}$ such that $K\cap Ka^n = \varnothing$ for $n>N$.
\end{proposition}

\begin{proof}
We first consider the case that $G$ is Hausdorff.

If we assume there exists a compact set $K$ such that $K\cap Ka^n\neq \varnothing$ for infinitely many $n$'s, then for those $n$'s, $a^n\in K^{-1}K$, which is impossible since there must admit a convergent subsequence in the compact set $K^{-1}K$. But this contradicts with the previous lemma, so the case of Hausdorff has been verified.

Now consider the general case. Let $G$ be a first countable locally compact group and $\pi :G\to \widetilde{G}$ be its Hausdorffication. For any compact set $K$ in $G$, $\pi(K)$ is compact in $\widetilde{G}$, so there exsit $N$ such that
\[\pi(K)\cap \pi(K)\pi(a)^n=\varnothing \text{ for } n>N\]
then
\begin{align*}
K\cap Ka^n &\subseteq (\pi ^{-1}\circ \pi)(K)\cap (\pi ^{-1}\circ \pi)(K)a^n\\
&= \pi ^{-1}(\pi(K))\cap \pi ^{-1}(\pi(K)\pi(a)^n)\\
&= \pi ^{-1}(\pi(K)\cap \pi(K)\pi(a)^n)\\
&= \pi ^{-1}(\varnothing )\\
&=\varnothing \text{ for } n>N.
\end{align*}
\end{proof}

\begin{theorem}
Let $G$ be a first countable locally compact group, then the following are equivalent:
\begin{enumerate}
\item
$a\in G$ is an aperiodic element.
\item
For any compact subset $K$ of $G$, there exists $N\in \mathbb{N}$ such that $K\cap Ka^n = \varnothing$ for $n>N$.
\end{enumerate}
\end{theorem}


\begin{definition}
\label{defterminal}
Let $G$ be a locally compact group and $a\in G$. We say $G$ has a terminal pair $(A,B)$ w.r.t. $a$ if there exists a pair of disjoint closed subsets $(A,B)$ of $G$ such that for any given compact subset $K$ in $G$, we have 
    \[Ka^{n}\subseteq A,\]
    \[Ka^{-n}\subseteq B\]
  for $n$ large enough.
\end{definition}
More intuitively, $a$ shifts any compact subset positively (resp. nagetively) into $A$ (resp. $B$).
\begin{example}
One of the simplest cases is $G=\mathbb{Z}\text{ or }\mathbb{R}$ and $a=1$, the terminal pair w.r.t. $1$ can be given by $(A,B)=([100,\infty),(-\infty,-100])$.
\end{example}

\begin{example}
\label{terexam}
Let $G$ be a general linear group $GL(n,\mathbb{C})$, $a\in G$ with some eigenvalue $\lambda $ such that $|\lambda |\neq 1$, then $G$ admits a terminal pair w.r.t. $a$.
\end{example}
\begin{proof}
Without loss of generality, we can assume $a$ is itself a Jordan form, since we can act a conjugate automorphism on $G$ as 
\[a= \begin{bmatrix}
\lambda & * & \cdots & 0 \\
\vdots & \ddots & \ddots & \vdots \\
0 & \cdots & \lambda '& * \\
0 & \cdots & 0 & \lambda ''
\end{bmatrix}\]

Notation: Let any $x\in G$, we write $x=[x_1|x_2|...|x_n]$, where $x_i\in \mathbb{C}^n$ are the column of $x$. (Note that $x_i$ will never be zero vector since $x$ is invertible.)

Consider the map $f:G\to \mathbb{R}$, $f(x)=ln\|x_1\|$. By the calculation, we have $f(xa^n)=f(x)+n*ln|\lambda|$.

Set $(A,B)=(f^{-1}([1,\infty )),f^{-1}((-\infty,-1]))$, so for any compact subset $K$ in $G$,
\[\inf_{x\in K} f(xa^n)=(\inf_{x\in K} f(x))+n*ln|\lambda|,\]
\[\sup_{x\in K} f(xa^n)=(\sup_{x\in K} f(x))+n*ln|\lambda|\]
which implies $(A,B)$ is a terminal pair w.r.t. $a$.
\end{proof}

\begin{proposition}
\label{teraper}
Let $G$ be a locally compact group and admit a terminal pair w.r.t. $a$, then $a$ is an aperiodic element.
\end{proposition}

\begin{proof}
Suppose $a$ is a periodic element but $G$ also admits a terminal pair w.r.t. $a$, then $G(a)a^n=G(a)$ for all $n\in\mathbb{Z}$, which implies that $A$ and $B$ are not disjoint.
\end{proof}

\begin{proposition}
\label{homoterminal}
Let $G,G'$ be locally compact groups, $G'$ admit a terminal pair w.r.t. $\phi (a)$, where $\phi :G \to G'$ is a continuous homomorphism, then $G$ also admits a terminal pair w.r.t. $a$. In other words, continuous homomorphisms pullback the terminal pair.
\end{proposition}
\begin{proof}
Let $(A_{G'},B_{G'})$ be a terminal pair w.r.t. $\phi (a)$. Equivalently, for any given compact subset $K'$ in $G'$, we have 
    \[K'\phi (a)^{n}\subseteq A_{G'},\]
    \[K'\phi (a)^{-n}\subseteq B_{G'}\]
  for $n$ large enough.

Set $(A_G,B_G)=(\phi ^{-1}(A_{G'}),\phi ^{-1}(B_{G'}))$. By continuity of $\phi$, $(A_G,B_G)$ is also a pair of disjoint closed sets in $G$. Now, for any given compact subset $K$ in $G$, $\phi (K)$ is also compact, so we have 
\vskip 1.1em
    \[Ka^{n}\subseteq \phi ^{-1}(\phi (Ka^{n}))=\phi ^{-1}(\phi (K)\phi (a)^{n})\subseteq \phi ^{-1}(A_{G'})=A_G,\]
    \[Ka^{-n}\subseteq \phi ^{-1}(\phi (Ka^{-n}))=\phi ^{-1}(\phi (K)\phi (a)^{-n})\subseteq \phi ^{-1}(B_{G'})=B_G\]
  for $n$ large enough.
\end{proof}

\begin{example}
Let $G=S^1\times\mathbb{R}$ and $a=(0,1)\in S^1\times\mathbb{R}$, where $S^1$ denotes the circle group. Then the natural quotient map $\phi :G \to \mathbb{R}$ sends $a$ to $1$, so terminal pair w.r.t. $a$ can be given by $(A,B)=(\phi^{-1}((-\infty,-100]),\phi^{-1}([100,\infty)))$.
\end{example}

\begin{proposition}
\label{terminalHausdorffication}
Let $G$ be locally compact group and $\pi :G\to \widetilde{G}$ be the Hausdorffication of $G$, then $(A,B)$ is a terminal pair w.r.t. $a$ iff $(\widetilde{A},\widetilde{B})=(\pi(A),\pi(B))$ is a terminal pair w.r.t. $\pi(a)$ with closed subsets $A,B$ of $G$.
\end{proposition}
\begin{proof}
Suppose $(A,B)$ is a terminal pair w.r.t. $a$. $\widetilde{A}$ and $\widetilde{B}$ are also disjoint closed follows by the closedness of $\pi$ and they are union of cosets of $\overline{\{e\}}$, immediately. $(\widetilde{A},\widetilde{B})$ is also a terminal pair w.r.t. $\pi(a)$ follows by there are one to one correspondence between the set of closed compact of $G$
and $\widetilde{G}$ (i.e. $\widetilde{K}=\pi(\pi^{-1}(\widetilde{K}))$). The other side follows from Proposition \ref{homoterminal}.
\end{proof}

Recall that any second countable locally compact Hausdorff groups are compatible with a proper ``right'' invariant metric $d$ \cite{plig}. In this metric space the Heine-Borel property holds. Therefore, $B_R(x):=\{y\in G|d(x,y)<R\}$ is a precompact open ball, and by right invariant $B_R(xa)=B_R(x)a$.

Note that \cite{plig} said that there's a proper ``left'' invariant metric say $d_L$, but it's easy to induce a proper ``right'' invariant metric $d$ by set $d(x,y):=d_L(x^{-1},y^{-1})$.

\begin{theorem}
\label{gotoinfty}
Let $G$ be a second countable locally compact Hausdorff group, then the following are equivalent:
\begin{enumerate}
\item
$a\in G$ is an aperiodic element.
\item
For any compact subsets $K$ and $K'$ in $G$, $d(Ka^{\ell},K')\to \infty$ as $\ell \to \infty$.
\end{enumerate}
\end{theorem}

\begin{proof}
(2$\Rightarrow $1). Choose $K=K'=\{e\}$, then $d(e,a^{\ell})\to \infty$ as $\ell \to \infty$. But if $a$ is periodic, then $G(a)$ is compact, hence bounded. This implies that $a^n$ are uniformly bounded, and there is a contradiction.

(1$\Rightarrow $2). Suppose $d(Ka^{\ell},K')$ is uniformly bounded by $C>0$, then for any $\ell\in\mathbb{N}$, there exists $k_{\ell}\in K$, $k'_{\ell}\in K'$ such that $d(k_{\ell}a^{\ell},k'_{\ell})=d(k_{\ell}a^{\ell}k'^{-1}_{\ell},e)<C$. Since the closed ball center at $e$ is compact, so there's a subsequence $\left\{\ell'\right\}$ , $k^{-1}\in K^{-1}$, $k'\in K'$ and an element $b$ in the ball such that 

\[k^{-1}_{\ell'}\to k^{-1},\]
\[k'_{\ell'}\to k',\]
\[k_{\ell'}a^{\ell'}k'^{-1}_{\ell'}\to b  \text{ as }  \ell' \to \infty.\]
Hence $a^{\ell'}\to k^{-1}bk' \text{ as }  \ell' \to \infty$, which contradicts with the Lemma \ref{discrete}, since $a^{\ell'}$ never converge in $G(a)\cong \mathbb{Z}$.
\end{proof}

\begin{theorem}
\label{equaperterminal}
Let $G$ be a second countable locally compact group, then the following are equivalent:
\begin{enumerate}
\item
$a\in G$ is an aperiodic element.
\item
For any compact subset $K$ of $G$, there exists $N\in \mathbb{N}$ such that $K\cap Ka^n = \varnothing$ for $n>N$.
\item
$G$ has a terminal pair w.r.t. $a$.
\end{enumerate}
\end{theorem}

\begin{proof}
By discussion above, we only need to prove the case (1$\Rightarrow $3).

We first consider the case that $G$ is Hausdorff.

Set 
\[J:=\left\{xa^n\in G|x\in G\text{ and } d(xa^n,e)\le d(xa^{n'},e) \text{ for all }{n'}\in \mathbb{Z}\right\}\]
\[N_x:=2\text{ } min\left\{N\in \mathbb{N}|2d(x,e)+2< d(xa^n,e)\text{ for all }n\in\mathbb{Z}\text{ with }|n|>N\right\}\]

The reason for the definition of $J$ is we want to simulate the special case $G=\mathbb{R}^2, a=(1,0)$. In this special case, the terminal pair w.r.t. $a$ can be set as $(\{(x,y)|x\geq 100\},\{(x,y)|x\le -100\})$. To deduce this, we need some ``sense'' like y-axis which is orthogonal to $a$, hence we define $J$ to be the simulation of y-axis.

Note that they are well-defined, since Theorem \ref{gotoinfty} implies that both $d(xa^n,e)$ and $d(xa^{-n},e)$ $\to \infty$ as $n\to \infty$. And for each $x\in G$, there exists $xa^n\in J$ (i.e. $G=\left\{xa^n|x\in J,n\in \mathbb{Z}\right\}$), and this follows from the same reason.

Note that $\mathfrak{B}:=\left\{B_{\frac{1}{4}}(xa^n)\right\}_{x\in J,n\in \mathbb{Z}}$ is a covering of $G$.

\textbf{Claim}: For any $x,y\in J$, $n>N_x$ and $m>N_y$, we have $d(xa^n,ya^{-m})>1$.

Suppose $1\geq d(xa^n,ya^{-m})=d(x,ya^{-n-m})=d(xa^{n+m},y)$.
Then
\begin{align*}
2+2d(x,e) &< d(xa^{n+m},e)\\
&\le d(xa^{n+m},y)+d(y,e)\\
&\le 1+d(y,e)\\
&\le 1+d(ya^{-n-m},e)\\
&\le 1+d(ya^{-n-m},x)+d(x,e)\\
&\le 2+d(x,e).\\
\end{align*}
\vskip -2.1em
There is a contradiction.

\textbf{Claim}: For any $x,y\in J$, $n>N_x$ and $m>N_y$, we have 
\[d(B_{\frac{1}{4}}(xa^n),B_{\frac{1}{4}}(ya^{-m}))>\frac{1}{4}.\]

Suppose $d(B_{\frac{1}{4}}(xa^n),B_{\frac{1}{4}}(ya^{-m}))\le\frac{1}{4}<\frac{1}{3}$.
Then there exists $x'\in B_{\frac{1}{4}}(xa^n)$ and $y'\in B_{\frac{1}{4}}(ya^{-m})$ such that $d(x',y')<\frac{1}{3}$, then we have $d(xa^n,ya^{-m})<\frac{1}{4}+\frac{1}{4}+\frac{1}{3}<1$, contradiction.

Now, we are ready to set our terminal pair. Set $(A,B):=(\overline{A'},\overline{B'})$, where 
\[A':=\bigcup_{x\in J,n>N_x}B_{\frac{1}{4}}(xa^n),\]
\[B':=\bigcup_{x\in J,n>N_x}B_{\frac{1}{4}}(xa^{-n}).\]
By claim, $d(A',B')\geq\frac{1}{4}>0$, so $A$ and $B$ are disjoint closed set.

Now we will show that for each compact set $K$ shift into $A$ and $B$. By compactness, there exist finitely many balls in $\mathfrak{B}$, $\left\{B_{\frac{1}{4}}(x_ia^{n_i})\right\}^{\ell}_{i=1}$ which covers $K$. Set\\
$N:=2\text{ }max_{1\le i\le\ell}\left\{|n_i-N_{x_i}|,|n_i+N_{x_i}|\right\}$, so $n_i+n>N_{x_i}$ and $n_i-n<-N_{x_i}$ for $n>N$. Then

\[Ka^n\subseteq \cup^{\ell}_{i=1} B_{\frac{1}{4}}(x_ia^{n_i})a^n= \cup^{\ell}_{i=1} B_{\frac{1}{4}}(x_ia^{n_i+n})\subseteq A,\]
\[Ka^{-n}\subseteq \cup^{\ell}_{i=1} B_{\frac{1}{4}}(x_ia^{n_i})a^{-n}= \cup^{\ell}_{i=1} B_{\frac{1}{4}}(x_ia^{n_i-n})\subseteq B\]
for $n>N$, so the case of Hausdorff has been verified.

Now consider the general case. Let $G$ be a second countable locally compact group and $\pi :G\to \widetilde{G}$ be its Hausdorffication, so $\pi (a)$ is also aperiodic, hence $\widetilde{G}$ has a terminal pair w.r.t. $\pi(a)$. By Proposition \ref{homoterminal}, we are done.

\end{proof}

\section{Existence of hypercyclic weighted translations} 
\label{sec:existence_of_hypercyclic_weighted_translations}
Now we would like to discuss how the existence of terminal pair affects the existence of hypercyclic weighted translation operators.
\begin{lemma}
\label{existencelemma}
Let $G$ be a second countable locally compact group and admit a terminal pair w.r.t. $a$, then there exists a weighted translation operator $T_{a,w}$ which satisfies the frequent hypercyclicity criterion on $L^p(G)$ for all $p\in [1,\infty)$, simultaneously.
\end{lemma}
\begin{proof}
We need to construct the weight $w$ by Urysohn's lemma and verify the frequent hypercyclicity criterion \cite[Theorem 9.9 and Proposition 9.11.]{Linear_chaos} directly.

Let $\pi :G\to \widetilde{G}$ be the Hausdorffication of $G$.

Recall that any locally compact Hausdorff group is normal \cite[Theorem 8.13]{E_Hewitt}, so $\widetilde{G}$ is normal. We use the notation in Proposition \ref{terminalHausdorffication}. Set $\widetilde{w}|_{\widetilde{A}}=2^{-1}$ and $\widetilde{w}|_{\widetilde{B}}=2$, by Urysohn's lemma $\widetilde{w}$ is a well-defined continuous function on $\widetilde{G}$ with the image lying in $[2^{-1},2]\subset(0,\infty)$, so we can induce $w:=\widetilde{w}\circ \pi$ with $w|_A=2^{-1}$, $w|_B=2$ and the image liying in $[2^{-1},2]\subset(0,\infty)$.

To verify the conclusion by frequent hypercyclicity criterion. We set \[X_0:=\{\text{bounded compact support functions on } G\}\] which is a dense subspace in $L^p(G)$ for all $p\in[1,\infty)$, and also set 
\[T=T_{a,w} \text{ and } S=T^{-1}_{a,w}=T_{a^{-1},w'},\]
where $w':=\left(T_{a^{-1}}w\right)^{-1}$.

Now given $\varphi\in X_0$. Let $K:=supp (\varphi)$ (note that $supp(\varphi(\cdot a^i))=Ka^{-i}$), let us first claim this:
\[\left\|T^n \varphi\right\|_{\infty} \text{ and } \left\|S^n \varphi\right\|_{\infty} \text{ decay to zero by exponential type.}\]
That is,
\begin{align*}
\left\|T^n \varphi\right\|_{\infty} \le C \gamma ^{-n}&,\\
\left\|S^n \varphi\right\|_{\infty} \le C \gamma ^{-n}&\text{  for some $C>0$, $\gamma>1$ and $n$ large enough.}
\end{align*}
We only need to prove the first case, since the second case is symmetric with the first case by replacing $A$ to $B$, $B$ to $A$, $a$ to $a^{-1}$ , $w$ to $w'$ and $T$ to $S$. Since $K$ is compact, there is a large number $N$ such that 
\[Ka^{n}\subseteq A,\]
\[Ka^{-n}\subseteq B\]
for $n\ge N$, then 

\begin{align*}
\left\|T^n \varphi\right\|_{\infty}&=\left\|\prod \limits^{n-1}_{i=0}w(xa^{-i}) \varphi(xa^{-n})\right\|_{\infty}\\
&=\left\|\prod \limits^{n-1}_{i=0}w(xa^{n-i}) \varphi(x)\right\|_{\infty}\\
&=\left\|\prod \limits^{n}_{j=1}w(xa^{j}) \varphi(x)\right\|_{\infty}\quad(j=n-i)\\
&\le \left\|\prod \limits^{n}_{j=1}w(xa^{j})|_K\right\|_{\infty}\|\varphi\|_{\infty}\\
&\le \left\|\prod \limits^{N}_{j=1}w(xa^{j})\right\|_{\infty}\left\|\varphi\right\|_{\infty}\prod \limits^{n}_{j=N+1}\left\|w(xa^{j})|_K\right\|_{\infty}\\
&= \left\|\prod \limits^{N}_{j=1}w(xa^{j})\right\|_{\infty}\left\|\varphi\right\|_{\infty}\prod \limits^{n}_{j=N+1}\left\|w(x)|_{Ka^{j}}\right\|_{\infty}\\
&\le \left\|\prod \limits^{N}_{j=1}w(xa^{j})\right\|_{\infty}\left\|\varphi\right\|_{\infty}\prod \limits^{n}_{j=N+1} 2^{-1}\quad (\text{since }Ka^j\subseteq A).
\end{align*}
So it has been verified that $\|T^n \varphi\|_{\infty}$ and $\|S^n \varphi\|_{\infty} $ decay to zero by exponential type.

Finally, we will show that 
\[\sum \limits^{\infty}_{n=0}T^n \varphi \text{ converges unconditionaly}\]
and
\[\sum \limits^{\infty}_{n=0}S^n \varphi \text{ converges unconditionaly}.\]
The same as above, we only need to verify the first case, and we will say they are actually converge absolutely. By aperiodicity of $a$, there is a large number $N_0$ such that 
\[K\cap Ka^{-n}=\varnothing \text{ for }n\ge N_0\]
then we can easily check $\mathcal{C}_\ell:=\{Ka^{-(N_0j+\ell)}\}_{j\in \mathbb{Z}}$ is a mutually disjoint collection for each $\ell=0,1,...,N_0-1$. So

\begin{align*}
\left\|\sum \limits^{\infty}_{n=0}T^n \varphi\right\|_p &\le \left\|\sum \limits^{N_0-1}_{\ell=0}\sum \limits^{\infty}_{j=0}\left|T^{N_0j+\ell} \varphi\right|\right\|_p\\
&\le \sum \limits^{N_0-1}_{\ell=0}\left\|\sum \limits^{\infty}_{j=0}\left|T^{N_0j+\ell} \varphi\right|\right\|_p.\\
\end{align*}
Note that the next step follows by $supp(T^n\varphi)\subseteq Ka^n$ and $\mathcal{C}_\ell$ are mutually disjoint collections and the monotone convergence theorem. Now, for any given $\ell=0,1,...,N_0-1$.
\begin{align*}
\left\|\sum \limits^{\infty}_{j=0}\left|T^{N_0j+\ell} \varphi\right|\right\|^p_p
&=\int {\left|\sum \limits^{\infty}_{j=0}T^{N_0j+\ell}{\varphi}\right|^p}\\
&=\int {\sum \limits^{\infty}_{j=0}\left|T^{N_0j+\ell}{\varphi}\right|^p}\text{ ($C_{\ell}$ are mutually disjoint)}\\
&=\sum \limits^{\infty}_{j=0}\int {\left|T^{N_0j+\ell}{\varphi}\right|^p}\\
&\le \sum \limits^{\infty}_{n=0}\left\|T^n \varphi\right\|^p_p\\
&= \sum \limits^{\infty}_{n=0}\int_{Ka^n} {\left|T^n{\varphi}\right|^p}\\
&\le \sum \limits^{\infty}_{n=0}\left|Ka^n\right| \left\|T^n \varphi\right\|^p_{\infty}\text{    ($\left|Ka^n\right|=\left|K\right|$)}\\
&= \left|K\right|\sum \limits^{\infty}_{n=0} \left\|T^n \varphi\right\|^p_{\infty} < \infty.\\
\end{align*}
Note that the last step follows by $\left\|T^n \varphi\right\|_{\infty}$ decay to zero by exponential type.

So the condition of frequent hypercyclicity criterion has been verfied.
\end{proof}
\begin{remark}
The existence of weighted translation operator $T_{a,w}$ is not unique. In fact, there are uncountable many weighted translations satisfying this lemma by setting $w|_A\equiv \alpha$ and $w|_B\equiv \beta$ for  $\alpha \in (0,1)$ and $\beta \in (1,\infty)$ whatever you like in the proof.
\end{remark}

\vskip 3em
Next theorem gives an answer to our main question.

\begin{theorem}
\label{main}
Let $G$ be a second countable locally compact group and $a$ be an aperiodic element in $G$, then there exists a weighted translation operator $T_{a,w}$ which is mixing, choatic and frequently hypercyclic on $L^p(G)$ for all $p\in [1,\infty)$, simultaneously.
\end{theorem}
\begin{proof}
By Theorem \ref{equaperterminal} and Lemma \ref{existencelemma}.
\end{proof}

\begin{example}
Let $G$ be an arbitrary Lie group and $a$ be an aperiodic element in $G$, then there exists a weighted translation operator $T_{a,w}$ which is mixing, choatic and frequently hypercyclic on $L^p(G)$ for all $p\in [1,\infty)$, simultaneously.
\end{example}

\begin{example}
Let $G$ be a general linear group $GL(n,\mathbb{C})$, then $a$ is a periodic element of $G$ iff $a$ is diagonalizable with each eigenvalue has norm $1$ (i.e. if $\lambda$ is a eigenvalue of $a$, then $|\lambda |=1$). In some case, it is hard to verify $G$ has a terminal pair w.r.t. $a$ by hand when $a$ is aperiodic. Specially, when 
$a= \begin{bmatrix}
-1 & \hbox{\hskip1.7ex}1 \\
 0 & -1 
\end{bmatrix}$, but by the discussion above, $G$ has a terminal pair w.r.t. $a$.
\end{example}

\begin{remark}
To explain why $a$ is periodic if and only if $a$ is diagonalizable with all eigenvalue has norm $1$, we only need to prove the case that if $a$ is nondiagonalizable then $a$ is aperiodic. Since other cases follow by Example \ref{terexam} and Proposition \ref{teraper}. The same as the Example \ref{terexam}, we can assume $a$ is itself a Jordan form: 
\[a= \begin{bmatrix}
\lambda & 1 & \cdots & 0 \\
 0 & \lambda & \cdots & \vdots \\
\vdots & \cdots & \ddots & * \\
0 & \cdots & 0 & \lambda '
\end{bmatrix}.\]
Notation: Let any $x\in G$, we write $x=[x_1|x_2|...|x_n]$, where $x_i\in \mathbb{C}^n$ are the column of $x$. (Note that $x_i$ will never be zero vector since $x$ is invertible.)

Consider the map $f:G\to \mathbb{R}$, $f(x)=\left|ln\|x_2\|\right|$, by the calculation, we have $f(a^n)=\left|(n-1)ln|\lambda |+ln|n|+\frac{1}{2}ln |1+\frac{|\lambda |^2}{n^2}|\right|\to \infty$ as $n\to \infty$ whatever $\lambda$ might be ($\lambda$ never be zero since $a\in GL(n,\mathbb{C})$). Suppose $a$ is periodic, then $f(G(a))$ is compact in $\mathbb{R}$, a contradiction as $\{f(a^n)\}$ is unbounded in $\mathbb{R}$.

\end{remark}

\newpage
\phantomsection
\addcontentsline{toc}{section}{Summary}
\noindent
\textbf{Summary.}

Given an aperiodic element $a$ in $G$. In order to find an expilcit form of a hypercyclic weighted translations associated to $a$. We can first use Proposition \ref{homoterminal} or similar techniques to find a expilcit terminal pair $(A,B)$ w.r.t. to $a$ by pullback argument. And try to contruct a expilcit continuous function $w$ such that $w|_A=2^{-1}$, $w|_B=2$ and the image liying in $[2^{-1},2]\subset(0,\infty)$. Then the operator $T_{a,w}$ will satisfies the frequent hypercyclicity criterion on $L^p(G)$ for all $p\in [1,\infty)$, simultaneously, as Lemma \ref{existencelemma} says.

\begin{example}
Let $G=GL(n,\mathbb{C})$ and $a\in G$ with some eigenvalue $\lambda $ such that $|\lambda |> 1$. Write
\[a= P
\begin{bmatrix}
\lambda & * & \cdots & 0 \\
\vdots & \ddots & \ddots & \vdots \\
0 & \cdots & \lambda '& * \\
0 & \cdots & 0 & \lambda ''
\end{bmatrix}
P^{-1},\]
for some $P\in G$. It is an aperiodic element by Example \ref{terexam} and Proposition \ref{teraper}.
Then $(A,B)$ will be a expilcit terminal pair w.r.t. to $a$, where
\[A:=\left\{PxP^{-1}|\|x_1\|>=2\right\}\]
and
\[B:=\left\{PxP^{-1}|\|x_1\|<=\frac{1}{2}\right\}.\]
(Note that $x_1$ means the first column of $x$.)

Set
\[w(PxP^{-1})=
\begin{cases}
\frac{1}{2} & \text{ if } \|x_1\|>=2 \\ 
2 & \text{ if } \|x_1\|<=\frac{1}{2} \\
\|x_1\|^{-1} & \text{ others.}
\end{cases}
\]
Obviously, $w$ is a continue function.
Then as Lemma \ref{existencelemma} says. The operator $T_{a,w}$ will satisfies the frequent hypercyclicity criterion on $L^p(G)$ for all $p\in [1,\infty)$, simultaneously.
\end{example}

\phantomsection
\addcontentsline{toc}{section}{Acknowledgements}
\noindent
\textbf{Acknowledgements.}
The author would like to thank Chung-Chuan Chen, Chiun-Chuan Chen and Chun-Wei Lee for their useful suggestions. Special thanks to Yi-Chiuan Chen who advised me to consider the common properties of all known examples of hypercyclic weighted translation operators when $a$ is aperiodic. His advice leads to this paper.

\bibliographystyle{abbrv}
\bibliography{On_aperiodicity_and_hypercyclic_weighted_translation_operators}

\begin{thebibliography}{1}

\bibitem{Birkhoff}
G.~Birkhoff.
\newblock A note on topological groups.
\newblock {\em Compositio Mathematica}, pages 427--430, 1936.

\bibitem{Chaotic_on_groups}
C.-C. Chen.
\newblock Chaotic weighted translations on groups.
\newblock {\em Archiv der Mathematik}, 97:61--68, 2011.

\bibitem{non-torsion}
C.-C. Chen.
\newblock Hypercyclic weighted translations generated by non-torsion elements.
\newblock {\em Archiv der Mathematik}, 101:135--141, 2013.

\bibitem{Hypercyclic_on_groups}
C.-C. Chen and C.-H. Chu.
\newblock Hypercyclic weighted translations on groups.
\newblock {\em Proceedings of the American Mathematical Society},
  139:2839--2846, 2011.

\bibitem{Linear_chaos}
K.-G. Grosse-Erdmann and A.~Peris.
\newblock Linear chaos.
\newblock {\em Springer, Universitext}, 2011.

\bibitem{E_Hewitt}
E.~Hewitt and K.~Ross.
\newblock Abstract harmonic analysis.
\newblock {\em Springer-Verlag, Heidelberg}, 1979.

\bibitem{plig}
R.~A. Struble.
\newblock Metrics in locally compact groups.
\newblock {\em Compositio Mathematica}, 28:217--222, 2006.

\bibitem{semigroups}
W.~S. Wolfgang~Desch and G.~F. Webb.
\newblock Hypercyclic and chaotic semigroups of linear operators.
\newblock {\em Cambridge University Press}, 17:793--819, 1997.

\end{thebibliography}

\vspace{.1in}
\end{document}